\documentclass{amsart}

\usepackage[utf8]{inputenc}

\usepackage[centertags]{amsmath}
\usepackage{MnSymbol}
\usepackage{amsthm}
\usepackage{amscd}
\usepackage{enumerate}
\usepackage[abbreviations]{foreign}
\usepackage[all,cmtip,2cell]{xy}
\UseAllTwocells
\usepackage{enumitem}
\usepackage{xcolor}
\definecolor{darkblue}{rgb}{0,0,0.3}
\usepackage[ocgcolorlinks,colorlinks=true, citecolor=darkblue, filecolor=darkblue, linkcolor=darkblue, urlcolor=darkblue]{hyperref}
\usepackage{hyphenat}

\usepackage[all,2cell]{xy}
\UseAllTwocells
\input xy
\xyoption{2cell}
\xyoption{all}
\usepackage{verbatim}
\usepackage{eucal}
\usepackage{stackrel}

\usepackage{tikz}
\usetikzlibrary{cd}
\usetikzlibrary{calc}
\usetikzlibrary{math}
\usetikzlibrary{arrows}
\usetikzlibrary{fit}
\usetikzlibrary{positioning}

\usetikzlibrary{decorations.pathmorphing, arrows.meta}

\tikzset{Rightarrow/.style={double equal sign distance,>={Implies},->},
RightarrowDashed/.style={double equal sign distance, dashed, >={Implies},->},
triple/.style={-,preaction={draw,Rightarrow,}},
tripleDashed/.style={-,dashed,preaction={draw,RightarrowDashed}},
quadruple/.style={preaction={draw,Rightarrow,shorten >=0pt},shorten >=1pt,-,double,double
distance=0.2pt}
quadrupleDashed/.style={preaction={draw,RightarrowDahed,shorten >=0pt},shorten >=1pt,-,dashed,double,double distance=0.2pt}}

\usepackage{color}
\usepackage{eucal}

\usepackage{abstract}

\newtheorem*{thm*}{Theorem}
\newtheorem{thm}{Theorem}[section]
\newtheorem{cor}[thm]{Corollary}

\newtheorem{lemma}[thm]{Lemma}
\newtheorem{prop}[thm]{Proposition}

\theoremstyle{definition}
\newtheorem{defi}[thm]{Definition}
\newtheorem{notate}[thm]{Notation}

\newtheorem{obs}[thm]{Observation}

\theoremstyle{remark}
\newtheorem{rmk}[thm]{Remark}
\newtheorem{example}[thm]{Example}

\newcommand{\tr}[2]{\mathchoice
	{#1\raise -1.8pt\vbox{\hbox{$\kern -.8pt/\mathsmaller{#2} $}}}
	{#1\raise -1.8pt\vbox{\hbox{$\kern -.8pt/#2$}}\kern .8pt}
	{#1\raise -1.8pt\vbox{\hbox{$\scriptstyle\kern -.8pt /#2$}}}
	{#1\raise -1.8pt\vbox{\hbox{$\scriptscriptstyle\kern -.8pt /#2$}}}}
\newcommand{\trbis}[2]{\mathchoice
	{#1\raise -1.8pt\vbox{\hbox{$\kern -.8pt\mathsmaller{/#2} $}}}
	{#1\raise -1.8pt\vbox{\hbox{$\kern -.8pt\mathsmaller{/#2}$}}\kern .8pt}
	{#1\raise -1.8pt\vbox{\hbox{$\scriptstyle\kern -.8pt /#2$}}}
	{#1\raise -1.8pt\vbox{\hbox{$\scriptscriptstyle\kern -.8pt /#2$}}}}

\newcommand{\A}{\mathcal{A}}
\renewcommand{\plus}[1]{\mathop{\amalg}\limits_{#1}}

\newcommand\nbd\nobreakdash

\newcommand{\s}{\mathcal{S}\mspace{-2.mu}\text{et}_{\Delta}}
\newcommand{\Ss}{\mathcal{S}\mspace{-2.mu}\text{et}_{\Delta}^{\,\mathrm{sc}}}

\newcommand{\ndef}{\emph}

\newcommand\pdfoo{\texorpdfstring{$\infty$}{oo}}
\newcommand\pdftwo{\texorpdfstring{$2$}{2}}

\newcommand{\Cat}{{\mathcal{C}\mspace{-2.mu}\mathit{at}}}
\newcommand{\nCat}[1]{{#1}\hbox{\protect\nbd-}\kern1pt\Cat}	

\newcommand{\msCat}{\mathcal{C}at^+_{\Del}}
\newcommand{\sCat}{\mathcal{C}at_{\Del}}

\newcommand{\C}{\mathcal{C}}
\newcommand{\D}{\mathcal{D}}

\newcommand{\E}{\mathcal{E}}

\newcommand{\M}{\mathcal{M}}

\renewcommand{\L}{\mathcal{L}}
\newcommand{\R}{\mathcal{R}}

\newcommand{\B}{\mathcal{B}}

\newcommand{\bS}{\mathbf{S}}

\newcommand{\ho}{\mathrm{ho}}

\DeclareMathOperator{\Hom}{Hom}

\DeclareMathOperator{\sca}{sc}

\DeclareMathOperator{\Psh}{PSh}

\def\Del{{\Delta}}

\def\Lam{{\Lambda}}

\renewcommand{\twocell}[5]{\ar@<#4ex>@{}[#1] \ar@<#4ex>@{=>}?(#3)+/d #5cm/;?(#3)+/u #5cm/^{#2}}

\makeatletter
\renewcommand{\tocsection}[3]{%
  \indentlabel{\@ifnotempty{#2}{\bfseries\ignorespaces#1 #2\quad}}\bfseries#3} 

\renewcommand{\tocsubsection}[3]{%
  \indentlabel{\@ifnotempty{#2}{\hspace{1.6em}\ignorespaces#1 #2\quad}}#3}
\makeatother

\numberwithin{equation}{section}

\setlist[itemize]{leftmargin=*}
\setlist[enumerate]{leftmargin=*}

\title[Cartesian factorization sistems and pointed cartesian fibrations]%
{Cartesian factorization sistems and pointed cartesian fibrations of \(\infty\)-categories}

\author{Edoardo Lanari}
\address{Institute of Mathematics CAS \\ \v{Z}itn\'a 25 \\115 67   Praha 1\\ Czech Republic}
\email{edoardo.lanari.el@gmail.com}
\urladdr{https://sites.google.com/view/edoardo-lanari}
\begin{document}

\maketitle
\begin{abstract}
	The goal of this paper is to prove an equivalence between the \((\infty,2)\)-category of \emph{cartesian} factorization systems on \(\infty\)-categories and that of \emph{pointed} cartesian fibrations of \(\infty\)-categories. This generalizes a similar result known for ordinary categories and sheds some light on the interplay between these two seemingly distant concepts.
\end{abstract}
\tableofcontents
\section*{Introduction}
This paper is part of an ongoing project aimed at understanding fibrations in higher categories. The main source of inspiration was \cite{RosickyTholenFact}, in which the authors analyse, among other things, the relation between fibrations of categories and orthogonal factorization systems on the corresponding total categories. More precisely, they prove the following result.
\begin{thm*}[Thm. 3.9, \cite{RosickyTholenFact}]
In a finitely complete category \(\C\), \((\E,\M)\) is a simple reflective factorization system on \(\C\) if and only if there exists a prefibration \(p\colon \C \to \B\) preserving the terminal object with\[\E=p^{-1}(\mathrm{Iso}\B), \quad \M=\mathrm{Cart}(p)\] where \(\mathrm{Cart}(p)\) denotes the class of \(p\)-cartesian morphisms.
\end{thm*}
Our version of this for \(\infty\)-categories is given in Theorem \ref{thm:equiv}, which we anticipate here.
\begin{thm*}
	There is an equivalence of \(\infty\)-bicategories between \(\mathbf{Cart}_{\ast}\) and \(\textbf{Fact}_{cart}\), which sends an object \(p\colon \E \to \B\) in \(\mathbf{Cart}_{\ast}\) to \((\E,(S^p_L,S^p_R))\), where \(S^p_L\) is the class of maps inverted by \(p\) and \(S^p_R\) is the class of \(p\)-cartesian morphisms.
\end{thm*}
We have relaxed the completeness assumption but we instead consider the more natural notion of cartesian fibration rather than pre-fibration. The latter class  is defined (see Section 3.7 in \cite{RosickyTholenFact}) as that of functors \(p\colon \C \to \B\) such that for every object \(C\) in \(\C\), the induced funtor \(p_C\colon \C_{/C} \to \D_{/pC}\) admits a right adjoint whose corresponding monad is \emph{idempotent}. We will prove in Proposition \ref{prop:cart fib characterization} that cartesian fibrations of \(\infty\)-categories are characterized by a similar but stronger property, namely that such adjunctions exist and are localizations (in the sense that the right adjoint is fully faithful).

The paper is organized as follows. In Section 1 we recall the necessary background material that serves as foundation for what follows. In particular, we clarify what models for higher categories we use, and the relevant results available in the literature. We use scaled simplicial sets as model for \((\infty,2)\)-categories, which are connected via a Quillen equivalence to marked simplicial categories, which will also appear in this paper.

Section 2 is devoted to cartesian fibrations and localizations. Here, we prove a useful lemma of independent interest (see Lemma \ref{lem:construction of ladj}) which will allow us to extend certain assignments to left adjoints. Next, we characterize cartesian fibrations as maps inducing localizations at the level of slice categories (Proposition \ref{prop:cart fib characterization}).

In the next section, we introduce factorization systems and prove some facts which are stated without proof in the literature (see Proposition \ref{prop:equivalent def of fact syst} and \ref{prop:fact sys from cart fib}). Next, we consider localizations on \(\infty\)-categories with a terminal object which are induced by factorization systems, and we prove stability of factorization systems under the formation of slice categories.

In the fourth and final section, the \((\infty,2)\)-categories of cartesian factorization systems and pointed cartesian fibrations are introduced. We denote them, respectively, by \(\mathbf{Fact}_{cart}\) and \(\mathbf{Cart}_{\ast}\). Here, after a careful analysis of their mapping \(\infty\)-categories, we establish the equivalence in Theorem \ref{thm:equiv}.
\section*{Acknowledgements}
The author gratefully acknowledges the support of Praemium Academiae of M.~Markl and RVO:67985840. A special thanks also goes to my friends and colleagues Ivan Di Liberti and Andrea Gagna for commenting on an early draft of this work, and to John Bourke for having directed my attention to the paper \cite{RosickyTholenFact}.
\section{Preliminaries}
In what follows, we will switch between \(\infty\)-categorical language and ordinary one quite freely. If the context allows for ambiguity we will specify, for instance, if (co)limits are intended as ordinary ones or \(\infty\)-categorical ones (e.g. homotopy (co)limits in some model categorical presentation). Also, the term \emph{\(\infty\)-groupoid} will be used to mean an object in the \(\infty\)-category of homotopy types, without any precise model of this in mind (but rather just the \(\infty\)-category obtained as the free-cocompletion of the terminal one). When we actually mean Kan complex, this will be made clear.
\subsection{\(\infty\)-groupoids and \(\infty\)-categories}
We will denote by \(\s\) the category of simplicial sets. We will employ the standard notation \(\Del^n \in \s\) for the \(n\)-simplex, and for \(\emptyset \neq S \subseteq [n]\) we write \(\Del^S \subseteq \Del^n\) the \((|S|-1)\)-dimensional face of \(\Del^n\) whose set of vertices is \(S\). For \(0 \leq i \leq n\), we will denote by \(\Lam^n_i \subseteq \Del^n\) the \(i\)'th horn in \(\Del^n\), that is, the subsimplicial set spanned by all \((n-1)\)-dimensional faces containing the \(i\)'th vertex. 
By an \(\infty\)-category we will always mean a \emph{quasi-category}, \ie a simplicial set \(X\) which admits extensions for all inclusions 
\(\Lambda^n_i\hookrightarrow\Delta^n\), for all \(n > 1\) and all \(0 < i < n\) (known as \emph{inner} horns). If an \(\infty\)-category \(X\), in addition, admits extensions for \(\Lambda^n_0\hookrightarrow \Delta^n\) and \(\Lambda^n_n\hookrightarrow \Delta^n\), then it is called a \emph{Kan complex}. These will be our favourite models for \(\infty\)-categories and \(\infty\)-groupoids, respectively.

Given an \(\infty\)-category \(\C\), we have several (equivalent) models for the \(\infty\)-groupoid of morphisms \(\C(x,y)\) between a pair of objects \(x,y\) in \(\C\). These exhibit \(\C\) as \emph{weakly enriched} over \(\infty\)-groupoids.
\begin{prop}
Let \(\C\) be an \(\infty\)-category. The following simplicial sets are all equivalent Kan complexes:
\begin{itemize}
\item the simplicial set \(\mathrm{map}_{\C}(x,y)\) defined by the following pullback square:
\[\begin{tikzcd}
\mathrm{map}_{\C}(x,y) \ar[d] \ar[r] & \C^{\Delta^1}\ar[d,"(\pi_0{,}\pi_1)"]\\
\Delta^0\ar[r,"\{(x{,}y)\}"]& \C\times \C
\end{tikzcd}\] where \((\pi_0,\pi_1)\) is the map induced by the inclusion \(\partial\Delta^1 \to \Delta^1\).
\item the simplicial set \(\mathrm{map}^{\triangleleft}_{\C}(x,y)\), whose set of \(n\)-simplices corresponds to \[\{\alpha\colon \Delta^{n+1}\to \C \ \vert \ \alpha(0)=x,\alpha_{\vert \Delta^{\{1,\ldots,n+1\}}}=\sigma(y)\}\] i.e. their restrictions to the \(d^0\)-face are degenerate at \(y\).
\item the simplicial set \(\mathrm{map}^{\triangleright}_{\C}(x,y)\), whose set of \(n\)-simplices corresponds to \[\{\alpha\colon \Delta^{n+1}\to \C \ \vert \ \alpha(n+1)=y,\alpha_{\vert \Delta^{\{0,\ldots,n\}}}=\sigma(x)\}\] i.e. their restrictions to the \(d^{n+1}\)-face are degenerate at \(x\).
\end{itemize}
\end{prop}
We will denote the \(\infty\)-groupoid represented (up to equivalence) by any of these Kan complexes by \(\C(x,y)\).

 Given an \(\infty\)-category \(X\), we will denote its homotopy category by \(\ho(X)\). This is the ordinary category having as objects the 0-simplices of \(X\), and as morphisms \(x \rightarrow y\) the set of equivalence classes of 1-simplices \(f\colon x \rightarrow y\) of \(X\) under the equivalence relation generated by identifying \(f\) and \(f'\) if there is a 2-simplex \(H\) in \(X\) with \( H|_{\Del^{\{1,2\}}}=f, \ H|_{\Del^{\{0,2\}}}=f'\) and \(H|_{\Del^{\{0,1\}}}\) degenerate on~\(x\). We recall that the functor \(\ho\colon \nCat{\infty}\rightarrow \nCat{1}\) is left adjoint to the ordinary nerve functor.
\begin{defi}
Let \(f\colon\C \to \D\) be a map of \(\infty\)-categories. Then we say \(f\) is:
\begin{itemize}
\item \emph{essentially surjective} if \(\ho(f)\colon\ho(\C)\to\ho(\D)\) is an essentially surjcetive functor between ordinary categories.
\item  \emph{fully faithful} if the induced map \(f_{x,y}\colon\C(x,y)\to \D(fx,fy)\) is an equivalence of \(\infty\)-groupoids for every pair of objects \((x,y)\) in \(\C\).
\end{itemize}
\end{defi}
Just like for ordinary category theory, we have the following useful result.
\begin{thm}[Thm. 3.9.7, \cite{CisinskiHigherCats}]
	\label{thm:ess surj+ff=eq}
A functor \(f\colon \C \to \D\) between \(\infty\)-categories is an equivalence if and only if it is essentially surjective and fully faithful.
\end{thm}
As far as the basics of limits and colimits in \(\infty\)-categories that we use here, we refer the reader to Chapter 4 of \cite{HTT} for the relevant terminology and results.

Finally, if \(\C\) is an \(\infty\)-category, then we say that a subobject \(\A\subset \C\) is a \emph{full subcategory of \(\C\) spanned by a set of objects \(A\)} if there exists a set of vertices \(A\subset \C_0\) of \(\C\) such that the simplicial set \(\A\) consists of all the simplices in \(\C\) whose vertices belong to \(A\). When this is the case, it is clear that \(\A\) is itself an \(\infty\)-category and the natural inclusion \(\A \to \C\) is fully faithful.
\subsection{Marked simplicial sets and marked-simplicial categories}

Our standard reference for marked simplicial sets is Chapter 3 of \cite{HTT}, and for scaled simplicial sets and marked simplicial categories we refer the reader to \cite{LurieGoodwillie}. We will use such objects as models for higher categories, as precisely described in what follows.
\begin{defi}
	A \emph{marked simplicial set} is a pair \((X,E)\) where \(X\) is  simplicial set and \(M\) is a subset of the set of 1-simplices of \(X\), called \emph{marked} simplices, such that it contains the degenerate ones. A map of marked simplicial sets \(f\colon (X,E_X)\rightarrow (Y,E_Y)\) is a map of simplicial sets \(f\colon X \rightarrow Y\) satisfying \(f(E_X)\subseteq E_Y\).
\end{defi}

The category of marked simplicial sets will be denoted by \(\s^+\). 

\begin{notate}\label{no:marked}
	For simplicity, we will often speak only of the non-degenerate marked edges when considering a marked simplicial set. For example, if \(X\) is a simplicial set and \(E\) is any set of edges in \(X\) then we will denote by \((X,E)\) the marked simplicial set whose underlying simplicial set is \(X\) and whose marked edges are \(E\) together with the degenerate edges.  In addition, when there is no risk of ambiguity, we will omit the set of marked 1-simplices and just denote \((X,E)\) by \(X\).\end{notate}

\begin{rmk}
	The category \(\s^+\) of marked simplicial sets
	admits an alternative description, as the category of models of a limit sketch. In particular, it is a reflective localization of a presheaf category
	and it is a cartesian closed category.
\end{rmk}

\begin{thm}[\cite{HTT}]\label{t:marked-categorical}
	There exists a model category structure on the category \(\s^+\) of marked simplicial sets in which cofibrations are exactly the monomorphisms and the fibrant objects are marked simplicial sets \((X,E)\) in which \(X\) is an \(\infty\)-category and \(E\) is the set of equivalences of \(X\), \ie 1-simplices \(f\colon \Delta^1 \rightarrow X\) which are invertible in \(\ho(X)\). 
\end{thm}

\begin{rmk}
	Marked simplicial sets are a model for \((\infty,1)\)-categories. Because of the description of the fibrant objects in the model structure on \(\s^+\), we will often consider an \(\infty\)-category as a marked simplicial set, where we implicitly understand the marking as the one given by the equivalences.
\end{rmk}

\begin{defi}
	We let \(\msCat\) denote the category of categories enriched over marked simplicial sets. We will refer to these as \emph{marked-simplicial categories}.
\end{defi}

By virtue of Proposition A.3.2.4 and Theorem A.3.2.24 of \cite{HTT}, the category \(\msCat\) is endowed with a model category structure in which the weak equivalences are the \ndef{Dwyer--Kan equivalences}. More explicitly, these are the maps \(f\colon \C \to \D\) which are
\begin{itemize}
	\item
	\emph{fully-faithful:} in the sense that the maps \(f_*\colon\C(x,y)\rightarrow \D(f(x),f(y))\) are marked categorical equivalences;
	\item
	\emph{essentially surjective:} in the sense that the functor of ordinary 
	categories given by \(f_*\colon \mathbf{ho}(\C)\rightarrow\mathbf{ho}(\D)\) is essentially surjective, where for a marked-simplicial category \(\E\) we denote by \(\mathbf{ho}(\E)\) the category whose objects are the objects of \(\E\) and such that \(\Hom_{\mathbf{ho}(\E)}(x,y) := [\Del^0,\C(x,y)]\) is the set of homotopy classes of maps from \(\Del^0\) to \(\C(x,y)\) with respect to the marked categorical model structure.
\end{itemize}
We also note that the trivial fibrations in \(\msCat\) are the maps \(f\colon\C \rightarrow \D\) which are surjective on objects and such that \(f_*\colon\C(x,y)\rightarrow \D(f(x),f(y))\) is a trivial fibration of marked simplicial sets for every pair of objects \((x,y)\) in \(\C\).

\begin{example}
Our main example of a marked simplicial category is that of \(\textsc{Cat}_{\infty}\), the marked simplicial category of \(\infty\)-categories. It is defined by having as objects \(\infty\)-categories, and the marked simplicial set \(\textsc{Cat}_{\infty}(X,Y)\) between two such objects is defined to be the \(\infty\)-category \(Y^X\) with marking given by the equivalences. The rest of the structure is defined in the obvious way.
\end{example}
\subsection{Scaled simplicial sets and \pdfoo-bicategories}

We now introduce scaled simplicial sets, which form another model for the theory of (\pdfoo,\pdftwo)-categories.

\begin{defi}[\cite{LurieGoodwillie}]
	A \emph{scaled simplicial set} is a pair \((X,T)\) where \(X\) is  simplicial set and \(T\) is a subset of the set of 2-simplices of \(X\), called the subset of \emph{thin} simplices, containing the degenerate ones. A map of scaled simplicial sets \(f\colon (X,T_X)\rightarrow (Y,T_Y)\) is a map of simplicial sets \(f\colon X \rightarrow Y\) satisfying \(f(T_X)\subset T_Y\).
\end{defi}

The category of scaled simplicial sets will be denoted by \(\Ss\). 

\begin{defi}\label{d:flat-sharp}
	Given a simplicial set \(X\) we will denote by \(X_{\flat} = (X,\deg_2(X))\) the scaled simplicial consisting of \(X\) with only degenerate triangles as thin \(2\)-simplices, and by \(X_{\sharp} = (X,X_2)\) the scaled simplicial set consisting of \(X\) with all triangles thin.
\end{defi}

\begin{rmk}\label{alternative def}
	The category \(\Ss\) admits an alternative description, as the category of models of a limit sketch. In particular, it is a reflective localization of a presheaf category. In fact, we can define a category \(\Delta_{\sca}\) having as set of objects the set \(\{[k]\}_{k\geq 0}\cup \{[2]_t\}\), obtained from \(\Delta\) by adding an extra object and maps \([2]\rightarrow[2]_t, \ \sigma^i_t\colon [2]_t\rightarrow [1]\) for \(i=0,1\) satisfying the obvious relations. The category \(\Ss\) is then the reflective localization of the category of presheaves \(\Psh(\Delta_{\sca})\) 
	(of sets) at the arrow \([2]_t\plus{[2]}[2]_t \rightarrow [2]_t  \), where we have identified an object of \(\Delta_{\sca}\) with its corresponding representable presheaf. Equivalently, the local objects are those presheaves \(X\colon\Delta_{\sca}^{\mathrm{op}}\rightarrow\mathbf{Set}\) for which \(X([2]_t)\rightarrow X([2])\) is a monomorphism.
	In particular, the category \(\Ss\) is cartesian closed and it is easy to check
	that the reflector functor \(\Psh(\Delta_{\sca}) \to \Ss\) preserves monomorphisms and commutes with finite products.
\end{rmk}

\begin{notate}\label{no:scaled}
	For simplicity, we will often speak only of the non-degenerate thin 2-simplices when considering a scaled simplicial set. For example, if \(X\) is a simplicial set and \(T\) is any set of triangles in \(X\) then we will denote by \((X,T)\) the scaled simplicial set whose underlying simplicial set is \(X\) and whose thin triangles are \(T\) together with the degenerate triangles. If \(L \subseteq K\) is a subsimplicial set then we use \(T|_L := T \cap L_2\) to denote the set of triangles in \(L\) whose image in \(K\) is contained in \(T\). 
\end{notate}

\begin{defi}
	\label{anodyne defi}
	Let \(\bS\) be the set of maps of scaled simplicial sets consisting of:
	\begin{enumerate}
		\item \label{item:scaled_anodyne_i} inner horns inclusions \(\bigl(\Lam^n_i,\{\Del^{\{i-1,i,i+1\}}\}|_{\Lam^n_i}\bigr)
		\subseteq \bigl(\Del^n,\{\Del^{\{i-1,i,i+1\}}\}\bigr)\) 
		for \(n \geq 2\) and \(0 < i < n\).		
		\item \label{item:scaled_anodyne_ii} the map
		\( (\Delta^4,T)\rightarrow \bigl(\Delta^4,T\cup \{\Delta^{\{0,3,4\}}, \ \Delta^{\{0,1,4\}}\}\bigr)\), where
		\[T:=\bigl\{\Delta^{\{0,2,4\}}, \ \Delta^{\{ 1,2,3\}}, \ \Delta^{\{0,1,3\}}, \ \Delta^{\{1,3,4\}}, \ \Delta^{\{0,1,2\}}\bigr\};\]
		\item \label{item:scaled_anodyne_iii} the set of maps \(\bigl(\Lambda^n_0\plus{\Delta^{\{0,1\}}}\Delta^0,\{\Delta^{\{0,1,n\}}\}|_{\Lam^n_0}\bigr)
		\rightarrow \bigl(\Delta^n\plus{\Delta^{\{0,1\}}}\Delta^0,\{\Delta^{\{0,1,n\}}\}\bigr)\) for \(n\geq 2\).
	\end{enumerate}
\end{defi}
We call \(\bS\) the set of \ndef{generating anodyne morphisms}, and its saturation is the class of \ndef{(scaled) anodyne maps}.
\begin{rmk}
	\label{rmk:j's are anod}
	As observed in Remark 3.1.4 of \cite{LurieGoodwillie}, the inclusions of scaled simplicial sets \(j_i\colon (\Del^3,T_i)\to \Del^3_{\sharp}\), for \(i=1,2\), where \(T_i\) is the collection of 2-simplices of \(\Del^3\) containing the \(i\)'th vertex, 
	are both anodyne. 
\end{rmk}
\begin{defi}\label{d:scaled-fib}
	We will say that a map of scaled simplicial sets \(X \to Y\) is a \emph{scaled fibration} if it has the right lifting property with respect to scaled anodyne maps.
\end{defi}
The following class of objects will be of particular interest to us. 
\begin{defi}[\cite{LurieGoodwillie}]
	A \emph{weak \(\infty\)-bicategory} is a scaled simplicial set \(\C\) which admits extensions along all scaled anodyne maps. 
\end{defi}

We observe that the map in the second point of Definition \ref{anodyne defi} ensures that marked 2-simplices of weak \(\infty\)-bicategories satisfy a \emph{saturation} property, while the first set guarantees, among other things, that the subobject of a weak \(\infty\)-bicategory \(X\) spanned by those \(n\)-simplices whose 2-dimensional faces are thin is an \(\infty\)-category.

%, which we call the \emph{core} \(\infty\)-category of \(X\). This will be denoted by \(X^{\thi}\).

%\begin{defi}
%	Let \(\C\) be a weak \(\infty\)-bicategory. We will say that an edge \(e\colon x \to y\) in \(\C\) is \emph{invertible} if it is invertible when considered in the core \(\infty\)-category \(\C^{\thi}\) of \(\C\), that is, if the corresponding arrow in the homotopy category \(\ho(\C^{\thi})\) is an isomorphism. In this case we will also say that \(e\colon x \to y\) is \emph{en equivalence} in \(\C\).
%	
%	If \(X\) is an arbitrary scaled simplicial set then we say that an edge in \(X\) is invertible if its image in \(\C\) is invertible for any scaled anodyne \(X \hrar \C\) with \(\C\) a weak \(\infty\)-bicategory (this does not depend on the choice of such a \(\C\)).
%\end{defi}
%
%\begin{rmk}
%	More explicitly, if \(\C\) is a weak \(\infty\)-bicategory then \(e\colon x \to y\) is invertible in \(\C\) if and only if there exist triangles of the form
%	\[\begin{tikzcd}[column sep=small]
%	x\ar[rr,equal]\ar[dr,"e"{swap}] & {}& x & & y\ar[rr,equal]\ar[dr,"f"{swap}] &{} & y \\
%	&y \ar[u,phantom, "\simeq"]\ar[ur,"f"{swap}]&&&& x \ar[u,phantom, "\simeq"] \ar[ur,"e"{swap}]
%	\end{tikzcd}\]
%	Indeed, this condition clearly implies that \(e\) corresponds to an isomorphism in \(\ho(\C^{\thi})\), and the implication in the other direction follows by applying Joyal's special outer horn theorem (\cite[Theorem 1.3]{JoyalQCatsAndKan}, see also~\cite[Proposition 1.2.4.3]{HTT}) to \(\C^{\thi}\).
%\end{rmk}

Next, we introduce a comparison between scaled simplicial sets and marked simplicial categories, in the form of an adjunction. Recall (Thm. 2.2.5.1, \cite{HTT}) that there is a Quillen equivalence of the form:
	\[ \xymatrixcolsep{1pc}
\vcenter{\hbox{\xymatrix{
			**[l]\s \xtwocell[r]{}_{\mathrm{N}}^{\mathfrak{C}
			}{'\perp}& **[r] \sCat}}}\]
		between the Joyal model structure on simplicial sets (see \cite{JoyalQCatsApplications}) and the model structure on simplicial categories (see Section 3.2 of the appendix of \cite{HTT}).
\begin{defi}
	The \ndef{scaled coherent nerve} functor \(\mathrm{N}^{\sca}\colon\msCat\rightarrow \Ss \) is defined by letting the underlying simplicial set of \(\mathrm{N}^{\sca}(\C)\) be the coherent nerve \(\mathrm{N}\C\) of the simplicially enriched category \(\C\), and setting its thin 2-simplices to be those maps \(f\colon\mathfrak{C}(\Delta^2)\rightarrow \C \) that send the unique non-degenerate 1-simplex of \(\mathfrak{C}(\Delta^2)(0,2)\) to a marked 1-simplex in \(\C(f(0),f(2))\).
\end{defi}

The functor \(\mathrm{N}^{\sca}\) admits a left adjoint \(\mathfrak{C}^{\sca}\colon \Ss \rightarrow \msCat\), whose explicit description can be found in Definition 3.1.30 of \cite{LurieGoodwillie}.

\begin{thm}[{\cite[Theorem 4.2.7]{LurieGoodwillie}}]
	\label{coherent nerve eq}
	There exists a model structure on the category \(\Ss\) of scaled simplicial sets, characterized as follows:
	\begin{itemize}
		\item a map \(f \colon X \rightarrow Y\) is a cofibration if and only if it is a monomorphism;
		\item a map \(f \colon X \rightarrow Y\) is a weak equivalence if and only \(\mathfrak{C}^{\sca}(f)\colon \mathfrak{C}^{\sca}(X)\rightarrow \mathfrak{C}^{\sca}(Y)\) is a weak equivalence in \(\msCat\).
	\end{itemize}
	Moreover, the adjoint pair 
	\[ \xymatrixcolsep{1pc}
	\vcenter{\hbox{\xymatrix{
				**[l]\Ss \xtwocell[r]{}_{\mathrm{N}^{\sca}}^{\mathfrak{C}^{\sca}
				}{'\perp}& **[r] \msCat}}}\]
	is a Quillen equivalence with respect to the model structures defined above.
\end{thm}
The model structure on \(\Ss\) described above is called the \emph{bicategorical} model structure.
One of the main results of \cite{GagnaHarpazLanariEquiv} can be summarized as follows.
\begin{thm}
The class of weak \(\infty\)-bicategories coincide with that of the fibrant objects in the bicategorical model structure on \(\Ss\). 
\end{thm}
We refer to the fibrant objects of the bicategorical model structure as \(\infty\)-\emph{bicategories}, for simplicity.
\section{Cartesian fibrations}
In this section we recall the notion of cartesian fibration of simplicial sets, and give a new characterization of it in terms of the existence of certain adjoint functors between \(\infty\)-categories. We fix a map \(p\colon X \to Y \) of simplicial sets which we will refer to throughout the whole section.

\begin{defi}
A \(1\)-simplex \(f\colon x' \to x\) in \(X\) is said to be \emph{\(p\)-cartesian} if for any given solid commutative square as depicted below, where \(n\geq 2\), the dashed lifting exists provided \(g(\Delta^{\{n-1,n\}})=f\).
\[
\begin{tikzcd}
\Lambda^n_n \ar[d,hook] \ar[r,"g"] & X\ar[d,"p"]\\
\Delta^n \ar[r,"l"]&Y
\end{tikzcd}\]
\end{defi}
Given the notion of cartesian 1-simplices, a cartesian fibration is essentially a map with enough cartesian lifts.
\begin{defi}
A map \(p\colon X \to Y\) is a \emph{cartesian fibration} if it is an inner fibration, and for every 1-simplex \(h\colon y \to p(x)\) there exists a \(p\)-cartesian 1-simplex \(\bar{h}\colon x' \to x\) with \(p(\bar{h})=h\)
\end{defi}
By dualizing the previous definitions we get the notion of \emph{\(p\)-cocartesian} 1-simplices and \emph{cocartesian fibrations}.

If the base of our map is an \(\infty\)-category we can give the following characterization.
\begin{prop}[\cite{HTT}, Prop.2.4.4.3]
	\label{prop:p-cart with hopull}
Let \(p\colon X \to Y\) be an inner fibration of \(\infty\)-categories and \(f\colon x' \to x\) a \(1\)-simplex of \(X\), then the following are equivalent:
\begin{enumerate}
\item \(f\) is \(p\)-cartesian.
\item for every vertex \(z\in X\) the following square is a pullback in the \(\infty\)-category of \(\infty\)-groupoids:
\begin{equation}
\label{eq:p-cart with hopull}
\begin{tikzcd}
X(z,x')\ar[r,"f\circ -"] \ar[d,"p_{z,x'}"{swap}] & X(z,x)\ar[d,"p_{z,x}"]\\
Y(pz,px') \ar[r,"p(f)\circ -"] & Y(pz,px)
\end{tikzcd}
\end{equation}
\end{enumerate}
\end{prop}
We now recall the fibrational definition of adjunction of \(\infty\)-categories.
\begin{defi}[\cite{HTT}, Def.5.2.2.1]
	Let \(\C,\D\) be \(\infty\)-categories. An \emph{adjunction} between \(\C\) and \(\D\) is a map \(q\colon\M\to \Delta^1\) which is simultaneously a cartesian and a cocartesian fibration, satisfying \(q^{-1}(0)\simeq \C\) and \(q^{-1}(1)\simeq \D\). By Proposition 5.2.1.4 of \cite{HTT}, we can associate functors \(f\colon \C \to \D\) and \(g\colon \D \to \C\) to such \(q\), and we say that \(f\) (resp. \(g\)) is \emph{left adjoint } to \(g\) (resp. \emph{right adjoint} to \(f\)). We denote this situation by \(f\dashv g\).
\end{defi}
The following result allows us to construct left adjoints given their assignments on objects, provided we have a suitable family of equivalences of hom-spaces.
\begin{lemma}
	\label{lem:construction of ladj}
	Let \(R\colon \D \to \C\) be a map of \(\infty\)-categories, and suppose given a map of vertices \(L_0\colon \C_0 \to \D_0\). Given equivalences of \(\infty\)-groupoids \(\Psi_{c,d}\colon\D\left(L_0c,d\right)\overset{\simeq}{\longrightarrow}\C(c,Rd)\) natural in \(d\in\D\), we can extend \(L_0\) to a map \(L\colon\C \to \D\) satisfying \(L\dashv R\).
\end{lemma}
\begin{proof}
	Let \(q\colon\M\to \Delta^1\) be the cartesian fibrations associated with \(R\), so that \(q^{-1}(0)\simeq \C\) and \(q^{-1}(1)\simeq \D \). Our task is to show \(q\) is also cocartesian. To obtain this, we have to exhibit, for every \(c\in\C\) a \(q\)-cocartesian morphism \(\eta\colon c\to d\) in \(\M\) with \(d\in \D\). Let \(d\stackrel{\text{def}}{=}L_0 c\), and choose a \(q\)-cartesian lift of the unique morphism \(0\to 1=q(x)\), which we denote by \(h\colon Rx \to x\). Consider the following diagram for an arbitrary \(x\in \D\).
	\[\begin{tikzcd}
	\D(L_0 c,x)\ar[r,"\simeq"]
	& \M(L_0 c,x)\ar[d,"q_{L_0 c,x}"] \ar[r,dotted]&\M(c,x)\ar[d,"q_{c,x}"]&\C(c,Rx)\ar[l,"\simeq"]\ar[d,"q_{c,Rx}"]\\
	\ast\ar[r,"\simeq"]&\Delta^1(1,1)\ar[r,"\simeq"]&\Delta^1(0,1)&\Delta^1(0,0)\ar[l,"\simeq"]
	\end{tikzcd}\] where the right-hand square is the pullback of \(\infty\)-groupoids induced by postcomposing by \(h\).
	
	The naturality of the composite map \(\D(L_0 c,x)\overset{\Psi_{c,x}}{\longrightarrow}\C(c,Rx)\overset{h\circ-}{\longrightarrow}\M(c,x)\), together with the observation that \(\D(L_0c,y)\cong\emptyset\) if \(y\in\C\), gives us a morphism \(f\colon c\to L_0c\) in \(\M\) by the Yoneda lemma, which renders the left-hand square a pullback. In fact, since \(-\circ f\colon\M(L_0c,c)\to \M(c,x)\) is a composite of equivalences, it must be an equivalence.
\end{proof}
\subsection{Localizations}
We give here a brief recap on localizations of \(\infty\)-categories, and we suggest the read of the relevant section in Chapter 5 of \cite{HTT} to the interested reader.

\begin{defi}
	A map \(f\colon\C \to \D\) between \(\infty\)-categories is a \emph{localization} if it admits a fully faithful right adjoint.
\end{defi}
Here is a useful characterization of localization functors.
\begin{prop}[Prop.5.2.7.4, \cite{HTT}]
Let \(\C\) be an \(\infty\)-category and let \(L\colon\C \to \C\) be a functor with essential image \(L\C\subset \C\). The following conditions are equivalent:
\begin{enumerate}
\item There exists a functor \(f\colon \C \to\D\) with a fully faithful right adjoint \(g\colon\D\to \C\) and an equivalence between \(g\circ f\) and \(L\).
\item When regarded as a functor from \(\C\) to \(L\C\), \(L\) is left adjoint to the inclusion \(L\C \hookrightarrow \C\).
\item There exists a natural transformation \(\alpha\colon\C \times \Delta^1 \to \C\) with \(\alpha\colon \mathrm{Id}_{\C}\to L\) such that, for every object \(C\in\C\), the morphisms \(L(\alpha_C),\alpha_{LC}\colon LC \to LLC\) are equivalences in \(\C\).
\end{enumerate}
\end{prop}
The next result show that this kind of localizations is indeed an instance of a more general one, in which we formally invert a given family of morphisms.
\begin{prop}[Prop.5.2.7.12, \cite{HTT}]
Let \(\C\) be an \(\infty\)-category and let \(L\colon \C \to \C\) be a localization functor with essential image \(L\C\). Let \(\mathcal{S}\) denote the collection of all morphisms \(f\) in \(\C\) such that \(Lf\) is an equivalence. Then, for every \(\infty\)-category \(\D\), composition with \(L\) induces a fully faithful functor 
%\[\psi\colon\mathrm{Fun}(L\C,\D)\to \mathrm{Fun}(\C,\D)\] 
\[\psi\colon \D^{L\C} \to \D^{\C}\]
whose essential image consists of all those map \(F\colon\C \to \D\) such that \(F(f)\) is an equivalence in \(\D\) for every morphisms \(f\) in \(\mathcal{S}\).
\end{prop}
We can now give the following characterization of cartesian fibrations of \(\infty\)-categories in terms of localization functors.
\begin{prop}
	\label{prop:cart fib characterization}
Let \(p\) be an isofibration of \(\infty\)-categories. Then \(p\) is a cartesian fibration if and only if the map \(p_x\colon X_{/x}\to Y_{/p(x)}\) admits a fully faithful right adjoint \(q_x\).
\end{prop}

\begin{rmk}
In other words, this result is saying that an isofibration of \(\infty\)-categories is a cartesian fibration if and only if the induced maps \(p_x\colon X_{/x}\to Y_{/p(x)}\) are localizations. Although the two statements are equivalent under the axiom of choice, it is clear from the proof that \(p\) is also endowed with a prescribed choice of \(p\)-cartesian lifts.
\end{rmk}

\begin{proof}
Suppose the right adjoints \(q_x\) exist for every vertex \(x\) in \(X\). Firstly, we observe that \(p_x\) is also an isofibration of \(\infty\)-categories. Therefore, we can assume, without loss of generality, that \(p_x\circ q_x=1_{Y_{/p(x)}}\). Indeed, we can consider the following commutative square:
\[\begin{tikzcd}
\{0\}\ar[d,hook] \ar[r,"q_x"] & {X_{/x}}^{Y_{/p(x)}}\ar[d,"(p_{/x})^*"]\\
J \ar[r,"\epsilon_x"] & {Y_{/p(x)}}^{Y_{/p(x)}}
\end{tikzcd}\] where \(\epsilon_x\) denotes the natural equivalence \(q_{x}\circ p_{x}\simeq 1_{Y_{/p(x)}}\) which features as the counit of the adjunction \(p_{x}\dashv q_{x}\), and \((p_{/x})^*\) is post-composition by \(p_x\). A lift for this square provides a map \(q'_x\colon Y_{/p(x)} \to X_{/x}\) which is equivalent to \(q_{x}\) and such that \(p_x\circ q'_x=1_{Y_{/p(x)}}\).

Now, given a 1-simplex \(h\colon y \to p(x)\) in \(Y\), set \(\bar{h}\stackrel{\text{def}}{=}q_x(h)\colon x' \to x\). By construction we have \(p(\bar{h})=h\), so let us show that \(\bar{h}\) is also \(p\)-cartesian. We have to show that the following square is a pullback in the \(\infty\)-category of \(\infty\)-groupoids:
\[\begin{tikzcd}
X(z,x')\ar[r,"q_x(h)\circ -"] \ar[d,"p_{z,x'}"{swap}] & X(z,x)\ar[d,"p_{z,x}"]\\
Y(pz,px') \ar[r,"h\circ -"] & Y(pz,px)
\end{tikzcd}\] this square can be modeled by a commutative square of Kan complexes where the vertical maps are Kan fibrations. Therefore, it is enough to show that the induced map at the level of fibers is a homotopy equivalence of Kan complexes. Equivalently, we can prove that there is an induced homotopy equivalence between the (homotopy) fibers of the horizontal maps. We have, by definition, the following (homotopy) pullback squares of simplicial sets, computed with respect to the Kan-Quillen model structure:
\[\begin{tikzcd}
X_{/x}(g,q_x(h)) \ar[r]\ar[d] &X(z,x')\ar[d,"q_x(h)\circ-"] &&& Y_{/p(x)}(p_x(g),h)\ar[r]\ar[d]& Y(pz,px')\ar[d,"h\circ-"]\\
\Delta^0 \ar[r,"\{g\}"] &X(z,x) &&&\Delta^0 \ar[r,"\{p(g)\}"] & Y(pz,px)
\end{tikzcd}\] We can thus conclude since the map between these two fibers is the adjunction equivalence \(Y_{/p(x)}(p_x(g),h)\simeq X_{/x}(g,q_x(h))\).

Conversely, assume \(p\) is a cartesian fibration. A choice of \(p\)-cartesian lifts for every \(1\)-simplex \(h\colon y\to p(x)\) as \(h,y\) and \(x\) vary provides a assignments on vertices of the form \(\left(q_x\right)_0\colon\left( Y_{/p(x)}\right)_0\to \left(X_{/x}\right)_0\). It follows from the previous paragraph that we also have homotopy equivalences of the form:
\(Y_{/p(x)}(p_x(g),h)\simeq X_{/x}(g,q_x(h))\) for every \(h\). This implies, thanks to Lemma \ref{lem:construction of ladj}, that we can extend \(\left(q_x\right)_0\) to a map \(q_x\colon Y_{/p(x)}\to X_{/x}\) which, in addition, must be a fully faithful right adjoint of \(p_x\).
\end{proof}
\begin{cor}
	\label{cor:cart fib are loc}
Let \(p\colon X \to Y\) be a cartesian fibrations of \(\infty\)-categories. Suppose \(X\) and \(Y\) admit a terminal object and that \(p\) preserves it. Then \(p\) is a localization functor.
\end{cor}
\begin{proof}
Let \(\ast_X\) (resp. \(\ast_Y\)) denote the terminal object of \(X\) (resp. \(Y\)). Then \(p\) is equivalent to the map \(p_{\ast_X}\colon X_{/\ast_X} \to Y_{/p(\ast_Y)}\), since \(p(\ast_X)\simeq \ast_Y\). This map is a localization functor thanks to the previous result, so we can conclude. 
\end{proof}

\section{Factorization systems}
In this section we recall the fundamental definitions for factorizations systems on \(\infty\)-categories, we prove that the two definitions available in the literature are equivalent and we show that cartesian fibrations always induce a factorization system on the total category.
\begin{defi}[\cite{HTT}]
	\label{defi:orthogonal maps}
Suppose given maps \(f\colon a \to b\) and \(g\colon x\to y\) in an \(\infty\)-category \(\C\). Then we say that \(f\) is \emph{left orthogonal} to \(g\) (and that \(g\) is \emph{right orthogonal} to \(f\)) if the following square is a pullback in the \(\infty\)-category of \(\infty\)-groupoids.
\[\begin{tikzcd}
\C(b,x)\ar[r,"g\circ -"] \ar[d,"-\circ f"{swap}] & \C(b,y)\ar[d,"-\circ f"]\\
\C(a,x) \ar[r,"g\circ-"] & \C(a,y)
\end{tikzcd}\] If this is the case, then we denote this relation by \(f\perp g\).
\end{defi}

\begin{rmk}
Informally speaking, this definition is saying that for every given commutative square in \(\C\) of the form:
\[\begin{tikzcd}
a \ar[r] \ar[d,"f"{swap}] & x\ar[d,"g"]\\
b\ar[r] \ar[ur,"h",dotted]&y
\end{tikzcd}\] there exists a lift as indicated by the dotted arrow \(h\), which is unique up to a contractible space (i.e. \(\infty\)-groupoid) of choices.
\end{rmk}

To make this formal, let us introduce the next idea, which is due to Joyal (see \cite{JoyalNotesQCat}). We can view squares in \(\C\) with a given lift in a coherent manner as maps \(\Delta^3\to \C\). Since \(\Delta^3\cong \Delta^1\star \Delta^1\), we get a natural inclusion \(\Delta^1\times \Delta^1\hookrightarrow \Delta^3\) which picks out the commutative square forgetting the lift. We thus get an induced isofibration \(\C^{\Delta^3}\to \C^{\Delta^1\times \Delta^1}\). Consider the following diagram, where both square are pullbacks and the right-hand side vertical maps are obtained by restriction:
\[\begin{tikzcd}
\mathrm{Lift}(f,g) \ar[r]\ar[d] & \C^{\Delta^3}\ar[d]\\
\C^{\Delta^1}(f,g) \ar[r] \ar[d] & \C^{\Delta^1 \times \Delta^1}\ar[d]\\
\Delta^0 \ar[r,"\{(f{,}g)\}"] & \C^{\{0\}\times \Delta^1}\times \C^{\{1\}\times \Delta^1}
\end{tikzcd}\]
\begin{prop}
	\label{prop:equivalent def of fact syst}
The pair of maps \((f,g)\) in \(\C\) satisfies \(f\perp g\) if and only if the induced map \(\mathrm{Lift}(f,g)\to \C^{\Delta^1}(f,g)\) is a trivial fibration of simplicial sets.
\end{prop}
\begin{proof}
Since the map \(\mathrm{Lift}(f,g)\to \C^{\Delta^1}(f,g)\) is a Joyal fibration of \(\infty\)-categories, it is enough to show it is an equivalence.

The inclusion \(\mathsf{Sp}^3\hookrightarrow \Delta^3\) is an inner anodyne map, so that we have an equivalence of \(\infty\)-categories \(\C^{\Delta^3} \to \C^{\mathsf{Sp}^3}\) over \(\C^{\{0\}\times \Delta^1}\times \C^{\{1\}\times \Delta^1}\). By pulling back along the map \(\{(f,g)\}\colon\Delta^0\to \C^{\{0\}\times \Delta^1}\times \C^{\{1\}\times \Delta^1}\) we get an equivalence of \(\infty\)-categories of the form \(\mathrm{Lift}(f,g) \simeq \C(b,x)\). Therefore, we are left with proving that \(f\perp g\) if and only if \(\C^{\Delta^1}(f,g)\simeq \C(b,x)\). Thanks to Proposition 5.1 of \cite{GepnerHaugsengNikolausLax}, we can express the term on the left-hand side by means of an end, as follows:
\[\C^{\Delta^1}(f,g)\simeq \int_{x \in \Delta^1}\C(fx,gx)\simeq \C(a,x)\times_{\C(a,y)} \C(b,y)\]
Hence, \(\C^{\Delta^1}(f,g)\simeq \C(b,x)\) is equivalent to having the pullback square as in Definition \ref{defi:orthogonal maps}.
\end{proof}
Let us now introduce the notion of factorization system.
\begin{defi}[\cite{HTT}, Def.5.2.8.8 and \cite{JoyalNotesQCat}, Section 24]
A factorization system on an \(\infty\)-category \(\C\) consists of a pair \(\mathcal{L},\mathcal{R}\) of collection of morphisms in \(\C\) satisfying the following properties:
\begin{enumerate}
\item Both families \(\mathcal{L}\) and \(\mathcal{R}\) are closed under retracts.
\item Every morphism in \(\mathcal{L}\) is left orthogonal to every morphism in \(\mathcal{R}\) (a fact which we will denote by \(\mathcal{L}\perp \mathcal{R}\)).
\item For every morphism \(h\colon x\to z\) in \(\C\) there is a 2-simplex in \(\C\) of the form:
\[\begin{tikzcd}
x \ar[dr,"f"{swap}]\ar[rr,"h"]&& z\\
&y \ar[ur,"g"{swap}] &
\end{tikzcd}\]
with \(f\in \mathcal{L}\) and \(g\in \mathcal{R}\)
\end{enumerate}
\end{defi}
For elements of the theory of factorization systems on \(\infty\)-categories we suggest the read of Section 5.2.8 of \cite{HTT}. 

The (dual of the) following result is stated without proof as Example 5.2.8.15 in \cite{HTT}.
\begin{prop}
	\label{prop:fact sys from cart fib}
Let \(p\colon \E\to\B\) be a cartesian fibration of \(\infty\)-categories. Then we get an induced factorization system \((p_L,p_R)\) on \(\E\), where \(p_L\) is the class of morphisms which are sent to equivalences by \(p\), and \(p_R\) is the class of \(p\)-cartesian morphisms.
\end{prop}
\begin{proof}
The fact that \(p_L\) is closed under retract is obvious, and for \(p_R\) it follows from Proposition \ref{prop:p-cart with hopull} together with the fact that pullback squares are closed under retracts. Indeed, suppose \(f\) is a retract of a \(p\)-cartesian morphism \(g\). The square in \eqref{eq:p-cart with hopull} for \(f\) is a retract of the analogous square for \(g\), so the same holds for the comparison map \(\E(z,x')\to \E(z,x)\times_{\B(pz,px)}\B(pz,px')\), which must then be a retract of an equivalence. It follows that such map must be an equivalence as well, so \(f\) is \(p\)-cartesian.

Next, suppose \(f\in p_L\) and \(g\in p_R\). Consider the following cube (where we have omitted some of the obvious maps for sake of clarity):
\[
\begin{tikzcd}[row sep=scriptsize, column sep=scriptsize]
& \E(b,c) \arrow[dl,"p_{b,c}"{swap}] \arrow[rr,"-\circ f"] \arrow[dd,"g\circ-"{swap}] & & \E(a,c)\arrow[dl] \arrow[dd,"g\circ-"] \\
\B(pb,pc) \arrow[rr,"-\circ p(f)", crossing over] \arrow[dd,"p(g)\circ -"{swap}]  & &  \B(pa,pc) \\
& \E(b,d) \arrow[dl,"p_{b,d}"] \arrow[rr,"-\circ f"] & & \E(a,d) \arrow[dl,"p_{a,d}"] \\
\B(pb,pd) \arrow[rr,"-\circ p(f)"] & & \B(pa,pd) \arrow[from=uu, crossing over]\\
\end{tikzcd}\] The front face is a pullback since \(p(f)\) is an equivalence by assumption. The left-hand side and right-hand side faces are pullbacks since \(g\) is \(p\)-cartesian, so it follows that the back face must be a pullback as well, i.e. \(f\perp g\).

Finally, assume given a morphism \(h\colon a\to b\) in \(\E\). Let \(\bar{h}\colon a' \to b\) be a \(p\)-cartesian lift of \(p(h)\), and consider the following lifting problem:
\[\begin{tikzcd}
\Lambda^2_2 \ar[d,hook] \ar[r,"(h{,}\bar{h})"] & \E\ar[d,"p"]\\
\Delta^2 \ar[r,"s^0(p(h))"] & \B
\end{tikzcd}\]
Since \(\bar{h}\) is \(p\)-cartesian, we get a lift \(H\colon \Delta^2 \to \E\) which is easily seen to be the factorization of \(h\) we are looking for.
\end{proof}
The following result describes a pretty common situation in which a factorization system induces a localization functor.
\begin{prop}
	\label{prop: localization with 1}
Let \(\E\) be an \(\infty\)-category endowed with a factorization system \((\mathcal{L},\mathcal{R})\). Suppose \(\E\) admits a terminal object \(\ast\in \E\) and let \(L\E\) be the full subcategory of \(\E\) spanned by those objects \(x\) such that the map \(!_x\colon x \to \ast\) belongs to \(\mathcal{R}\). Then, we get a localization functor \(L\colon \E \to \E\) which exhibits \(L\E\) as a localization of \(\E\). Moreover, if \(\mathcal{L}\) has the two-out-of-three property, then this is the localization at the class of maps \(\mathcal{L}\).
\end{prop}
\begin{proof}
We begin by factoring the map \(!_x\colon x \to \ast\) as \(x\to L_0x \to \ast\), where \(l_x\colon x \to L_0 x\) belongs to \(\mathcal{L}\) and \(L_0 x \to \ast\) belongs to \(\mathcal{R}\). In this manner, we obtain an assignment on objects of the form \(L_0\colon\E_0 \to (L\E)_0\). In order to apply Lemma \ref{lem:construction of ladj}, we have to exhibit equivalences of the form \(\E(L_0 x,y)\simeq \E(x,y)\) for every \(y\in L\E\). Consider the following commutative square:
\[
\begin{tikzcd}
\E(L_0x,y)\ar[r,"-\circ l_x"]\ar[d,"!_y\circ -"{swap}] & \E(x,y)\ar[d,"!_y\circ -"] \\
\E(L_0x,\ast) \ar[r,"-\circ l_x"]& \E(x,\ast)
\end{tikzcd}
\] By assumption, it is a pullback of \(\infty\)-groupoids, since \(l_x\perp !_y\). Because \(\ast\) is a terminal object in \(\E\), the upper horizontal map must be an equivalence, which is natural in \(y\in L\E\), so we can conclude by applying Lemma \ref{lem:construction of ladj}.

Turning to the second point, suppose \(L(f)\) is an equivalence. This implies that we have a commutative diagram in \(\E\) of the form:
\[\begin{tikzcd}
x\ar[r,"f"]\ar[d,"l_x"{swap}] & y\ar[d,"l_y"]\\
Lx \ar[r,"L(f)"]& Ly
\end{tikzcd}\] Since \(L(f)\) is an equivalence by assumption, and since \(\mathcal{L}\) is assumed to have the two-out-of-three property, we get that \(f\) belongs to \(\mathcal{L}\), which concludes the proof.
\end{proof}

\begin{defi}
Let \((\E,(\mathcal{L},\mathcal{R}))\) be a factorization system on an \(\infty\)-category \(\E\) with terminal object. The localization functor \(L\colon \E \to L\E \) described in Proposition \ref{prop: localization with 1} is called \emph{simple} if a morphism \(f\colon x 	\to y\) in \(\E\) is in \(\mathcal{R}\) if and only if the naturality square depicted below is a pullback in \(\E\).
\[\begin{tikzcd}
x \ar[r,"f"] \ar[d,"l_x"{swap}] & y\ar[d,"l_y"]\\
Lx \ar[r,"L(f)"]& Ly
\end{tikzcd}\]
\end{defi}
In case the factorization system is induced by a cartesian fibration we have the following corollary.
\begin{cor}
\label{cor:cart fact are simple}
Let \(p\colon \E \to \B\) be a cartesian fibration between \(\infty\)-categories admitting a terminal object. If \(p\) preserves such terminal object, then the localization functor \(L\colon\E \to L\E\) (which exists thanks to Proposition \ref{prop:fact sys from cart fib} and Proposition \ref{prop: localization with 1}) is simple.
\end{cor}
\begin{proof}
Firstly, let us observe that the localization functor coincides with \(p\) itself, thanks to the abovementioned propositions. We will use the notation \(L\colon\E \to L\E\) to stress the fact that it is a localization.

 Suppose the naturality square 
\begin{equation}
\label{nat sq}
\begin{tikzcd}
x \ar[r,"f"] \ar[d,"l_x"{swap}] & y\ar[d,"l_y"]\\
Lx \ar[r,"L(f)"]& Ly
\end{tikzcd}
\end{equation}
is a pullback in \(\E\). Let us now consider the following cube in \(\E\) (where, for sake of clarity, we have omitted some of the arrows' labels), for any object \(a\) in \(\E\):
\[
\begin{tikzcd}[row sep=scriptsize, column sep=scriptsize]
& \E(a,x) \ar[dl,"l_x\circ -"{swap}] \ar[rr,"f\circ -"] \ar[dd] & & \E(a,y)\ar[dl] \ar[dd,"L_{a,y}"] \\
\E(a,Lx) \ar[rr, crossing over] \ar[dd,"\simeq"{swap}]  & &  \E(a,Ly) \ar[dd]\\
& L\E(La,Lx) \ar[dl,"\simeq"{swap}] \ar[rr,"-\circ f"] & & L\E(La,Ly) \ar[dl,"\simeq"] \\
L\E(La,Lx) \ar[rr,"L(f)\circ -"{swap}] & & L\E(La,Ly) \ar[from=uu, crossing over]\\
\end{tikzcd}\]
The top face is obtained by applying \(\E(a,-)\) to the square \(\eqref{nat sq}\), and is thus a pullback. The front face and the bottom one are pullbacks since they both have a pair of parallel maps which are equivalences, so the back face must be a pullback as well. This proves \(f\) is \(p\)-cartesian.

Conversely, assume \(f\) is \(p\)-cartesian. Then the square
\[\begin{tikzcd}
\E(a,x) \ar[r,"f\circ -"] \ar[d,"l_x\circ -"{swap}] & \E(a,y)\ar[d,"l_y\circ -"]\\
\E(a,Lx) \ar[r,"L(f)\circ -"]& \E(a,Ly)
\end{tikzcd}\] obtained by applying \(\E(a,-)\) to the square \eqref{nat sq} is equivalent to the square
\[\begin{tikzcd}
\E(a,x) \ar[r,"f\circ -"] \ar[d,"l_x\circ -"{swap}] & \E(a,y)\ar[d,"l_y\circ -"]\\
L\E(La,Lx) \ar[r,"L(f)\circ -"]& L\E(La,Ly)
\end{tikzcd}\] which is a pullback. It follows that the former must be a pullback as well, and since this holds for every \(a\) in \(\E\) we conclude that the square \eqref{nat sq} is indeed a pullback.
\end{proof}
We conclude this section with the following result, concerning the creation of factorization systems by suitable forgetful functors.
\begin{lemma}
	\label{lem: fact on slice}
Let \(\E\) be an \(\infty\)-category endowed with a factorization system \((\mathcal{L},\mathcal{R})\). Given any object \(x\in\E\), we get a factorization system \(({\mathcal{L}}^x,{\mathcal{R}}^x)\) on \(\E_{/x}\), defined by \({\mathcal{L}}^x=p^{-1}(\mathcal{L})\) and \({\mathcal{R}}^x=p^{-1}(\mathcal{R})\). 
\end{lemma}
In the previous situation, we say that \(({\mathcal{L}}^x,{\mathcal{R}}^x)\) is \emph{created} by the projection \(p\colon\E_{/x}\to \E\).
\begin{proof}
The stability under the formation of retracts and the factorization property are obvious consequences of the same properties for \((\mathcal{L},\mathcal{R})\). To show that \({\mathcal{L}}^x\perp {\mathcal{R}}^x\), we consider the following cube: 
\[
\begin{tikzcd}[row sep=scriptsize, column sep=scriptsize]
& \E_{/x}(p_B,p_C) \arrow[dl,"p_{B,C}"{swap}] \arrow[rr,"-\circ f"] \arrow[dd,"g\circ-"{swap}] & & \E_{/x}(p_A,p_C)\arrow[dl] \arrow[dd,"g\circ-"] \\
\E(B,C) \arrow[rr,"-\circ f", crossing over] \arrow[dd,"g\circ -"{swap}]  & &  \E(A,C) \\
& \E_{/x}(p_B,p_D) \arrow[dl,"p_{B,D}"] \arrow[rr,"-\circ f"] & & \E_{/x}(p_A,p_D) \arrow[dl,"p_{A,D}"] \\
\E(B,D) \arrow[rr,"-\circ f"] & & \E(A,D) \arrow[from=uu, crossing over]\\
\end{tikzcd}\] where \(f\colon p_A \to p_B\) belongs to \( {\mathcal{L}}^x\), \(g \colon p_C \to p_D\) belongs to\( {\mathcal{R}}^x\) and \(p_I\colon I \to X\) are objects in \(\E_{/x}\). It is easy to show that the front face, the left-hand side and the right-hand side ones are all pullback of \(\infty\)-groupoids, therefore the back face must be such as well. This implies \(f\perp g\) and concludes the proof.
\end{proof}
\section{The equivalence}
In this section we will identify suitable subcategories of, respectively, the \(\infty\)-bicategory of cartesian fibrations (where we let the base vary) and that of \(\infty\)-bicategories endowed with a factorization system, and we will prove that they are equivalent.

\begin{defi}
Let \(\textsc{Fact}\) be the marked simplicial category whose objects are pairs \((\E,(\mathcal{L},\mathcal{R}))\), where \(\E\) is an \(\infty\)-category endowed with a factorization system \((\mathcal{L},\mathcal{R})\), and whose mapping \(\infty\)-category \(\textsc{Fact}((\E,(\mathcal{L},\mathcal{R})),(\D,(S'_L,S'_R)))\) between two such objects is defined as the full subcategory of \(\textsc{Cat}_{\infty}(\E,\D)\) spanned by those maps \(f\colon \E \to \D\) satisfying \(f(\mathcal{L})\subset S'_L\) and \(f(\mathcal{R})\subset S'_R\). We let \(\textbf{Fact}\) be the scaled simplicial nerve \(\mathrm{N}^{sc}(\textsc{Fact})\) of the marked simplicial category \(\textsc{Fact}\). 
\end{defi}

Note that, since \(\textsc{Fact}\) is enriched over \(\infty\)-categories by definition, we have that \(\textbf{Fact}\) is an \(\infty\)-bicategory. 

We now want to isolate a specific subcategory of this \(\infty\)-bicategory we have just defined. To achieve this, we have to introduce a condition on the localizations induced at the level of slice \(\infty\)-categories.

Suppose given an object \((\E,(\mathcal{L},\mathcal{R}))\) in \(\textbf{Fact}\), such that \(\E\) admits a terminal object and \(\mathcal{L}\) has the two-out-of-three property. Thanks to Proposition \ref{prop: localization with 1}, we know that we get a localization functor \(L\colon \E \to L\E\), which is the localization with respect to the class of maps \(\mathcal{L}\). Moreover, \(\E_{/x}\) also has a terminal object, and thanks to Lemma \ref{lem: fact on slice} we thus get a localization map \(L_x\colon \E_{/x}\to (\mathcal{R})_{/x}\), where the \(\infty\)-category on the right is the full subcategory of \(\E_{/x}\) spanned by those objects \(f\colon e\to x\) with \(f\in \mathcal{R}\).
\begin{defi}
	\label{def:cartesian objects}
An object \((\E,(\mathcal{L},\mathcal{R}))\) in \(\textbf{Fact}\) is said to be \emph{cartesian} if \(\E\) admits a terminal object, \(\mathcal{L}\) has the two-out-three property and the restriction \[L_{\vert}\colon(\mathcal{R})_{/x} \to L\E_{/Lx}\] of the functor \(L\colon \E_{/x}\to L\E_{/Lx}\) is an equivalence of \(\infty\)-categories. We denote this last condition by \((\ast)\).
\end{defi}
\begin{rmk}
	\label{rmk: cond *}
The last condition in the previous definition can be rephrased as follows. Firstly, we observe that thanks to Theorem \ref{thm:ess surj+ff=eq}, it is equivalent to the fact that \(L_{\vert}\) is an essentially surjective and fully faithful functor. Essential surjectivity is easily shown to be equivalent to requiring that for every map \(g\colon Le \to Lx\) in \(\mathcal{R}\) we can find a map \(p\colon e' \to x\) in \(\mathcal{R}\) and an equivalence \(g \simeq L(p)\) over \(Lx\). 

Fully faithfulness translates into the fact that the map induced between the (homotopy) fibers of the vertical maps of the square depicted below at every element \(f\colon e \to x\) in \(\mathcal{R}\) (resp. \(L(f)\)) is an equivalence of \(\infty\)-groupoids, when \(g\) belongs to \( \mathcal{R}\). 
\begin{equation}
\label{ff}
\begin{tikzcd}
\E(e,e')\ar[d,"g\circ -"{swap}] \ar[r,"L_{e,e'}"]&L\E(Le,Le')\ar[d,"Lg\circ -"]\\
\E(e,x)\ar[r,"L_{e,x}"] & L\E(Le,Lx)
\end{tikzcd}
\end{equation}
\end{rmk}
We are now in a position to define one of the two \(\infty\)-bicategories of interest.
\begin{defi}
Let \(\textbf{Fact}_{cart}\) be the subobject of \(\textbf{Fact}\) spanned by those \(n\)-simplices whose 1-dimensional faces consist of maps \(f\colon (\E,(\mathcal{L},\mathcal{R}))\to(\D,(S'_L,S'_R))\) between cartesian objects, which preserve the terminal object. By definition, this also implies \(f(\mathcal{L})\subset S'_L\) and \(f(\mathcal{R})\subset S'_R\).
\end{defi}
It is not too hard to see that this is indeed a well-defined \(\infty\)-bicategory, an alternative description of which is given by taking the scaled simplicial nerve of \(\textsc{Fact}_{Cart}\), defined as the subcategory of \(\textsc{Fact}\) on the cartesian objects, where we restrict the mapping \(\infty\)-categories to maps preserving the terminal objects.

We now define a suitable \(\infty\)-bicategory of cartesian fibrations and isolate from it the subcategory of interest.
\begin{defi} 
Let \(\textsc{Cart}\) be the marked simplicial category whose objects are cartesian fibrations of \(\infty\)-categories and whose mapping \(\infty\)-category \(\textsc{Cart}(p,q)\) is given by the subspace of \(\textsc{Sq}(p,q)\) (defined by the ordinary pullback depicted below) on those maps \(X\to Y\) that send \(p\)-cartesian morphisms to \(q\)-cartesian morphisms.
	\[\begin{tikzcd}
	\textsc{Sq}(p,q) \ar[d] \ar[r] & \textsc{Cat}_{\infty}(W,Z)\ar[d,"-\circ p"]\\
	\textsc{Cat}_{\infty}(X,Y) \ar[r,"q\circ -"] & \textsc{Cat}_{\infty}(X,Z)
	\end{tikzcd}\]
Since cartesian fibrations are, in particular, fibrations in the Joyal model structure, we have that \(\textsc{Cart}\) is enriched over \(\infty\)-categories.

% We thus let \(\mathbf{Cart}\) be the \(\infty\)-bicategory given by \(\mathbf{Cart}\stackrel{\text{def}}{=}\mathrm{N}^{sc}\textsc{Cart} \).
%Let \(\mathbf{Cat}^{\Delta^1}_{\infty}\) be the \(\infty\)-bicategory of maps of \(\infty\)-categories. Given cartesian fibrations of \(\infty\)-categories \(p\colon X \to W\) and \(q\colon Y \to Z\), we define \(\mathbf{Cart}(p,q)\) as the full subcategory of \(\mathbf{Cat}^{\Delta^1}_{\infty}(p,q)\) spanned by those maps \((f,g)\) with \(f\colon X \to Y\) and \(g\colon W \to Z\) such that \(f\) sends \(p\)-cartesian simplices to \(q\)-cartesian simplices. We let \(\mathbf{Cart}\) be the (non-full) subcategory of \(\mathbf{qCat}^{\Delta^1}\) spanned by those \(n\)-simplices whose 1-dimensional faces belong to \(\mathbf{Cart}(p,q)\) for some cartesian fibrations \(p\) and \(q\) (it is easy to check this is well defined, essentially because cartesian simplices are closed under composition and equivalences). In particular, its objects are cartesian fibrations between \(\infty\)-categories.
\end{defi}
%\begin{rmk}
%It is easy to see that, since cartesian fibrations are, in particular, fibrations in the Joyal model structure, the \(\infty\)-category \(\mathbf{Cart}\) can equivalently be described as the coherent simplicial nerve of the simplicial category \(\textsc{Cart}\), whose objects are cartesian fibrations of \(\infty\)-categories and whose mapping space \(\textsc{Cart}(p,q)\) is given by the subspace of \(\textsc{Sq}(p,q)\) (defined by the ordinary pullback depicted below) on those maps \(X\to Y\) that send \(p\)-cartesian morphisms to \(q\)-cartesian morphisms.
%\[\begin{tikzcd}
%\textsc{Sq}(p,q) \ar[d] \ar[r] & \textsc{Cat}_{\infty}(W,Z)\ar[d,"-\circ p"]\\
%\textsc{Cat}_{\infty}(X,Y) \ar[r,"q\circ -"] & \textsc{Cat}_{\infty}(X,Z)
%\end{tikzcd}\]
%\end{rmk}
The cartesian fibration we will be interested in are introduced in the following definition.
\begin{defi}
A cartesian fibration \(p\colon X \to W\) is said to be \emph{pointed} if \(X\) and \(W\) have a terminal object and \(p\) preserves it.
\end{defi}
The next result simplifies the description of the mapping \(\infty\)-category \(\textsc{Cart}(p,q)\) in case \(p\) is a pointed cartesian fibration.
\begin{lemma}
	\label{lem:mapping space of special cart fib}
Given cartesian fibrations \(p\colon X \to W\) and \(q\colon Y \to Z\) between \(\infty\)-categories, with \(p\) pointed, we have an equivalence of \(\infty\)-categories between \(\textsc{Cart}(p,q)\) and the full subcategory of \(\textsc{Cat}_{\infty}(X,Y)\) spanned by those maps \(f\colon X\to Y\) that send \(p\)-cartesian morphisms to \(q\)-cartesian morphisms and such that if \(h\) is a morphisms in \(X\) inverted by \(p\) then \(f(h)\) is inverted by \(q\).
\end{lemma}
\begin{proof}
By Corollary \ref{cor:cart fib are loc}, we get that \(p\) must be a localization functor. In particular, it is the localization at all the maps inverted by \(p\), which means that we have an equivalence of \(\infty\)-categories of the form \(\textsc{Cat}_{\infty}(W,Z)\simeq \textsc{Cat}'_{\infty}(X,Z)\), where \(\textsc{Cat}'_{\infty}(X,Z)\) denotes the full subcategory of \(\textsc{Cat}_{\infty}(X,Z)\) spanned by those functors \(X\to Z\) that invert all the morphisms which are inverted by \(p\). This implies that \(\textsc{Sq}(p,q)\) is equivalent to the full subcategory of \(\textsc{Cat}_{\infty}(X,Y)\) spanned by those functors \(f\colon X\to Y\) such that \(q\circ f\) inverts all the morphisms inverted by \(p\). Under this identification, \(\textsc{Cart}(p,q)\) corresponds to the one in the statement.
\end{proof} 
We can now give the following definition, which isolates the subcategory of our interest.
\begin{defi}
Let \(\textsc{Cart}_{\ast}\) be the (non-full)  subcategory of the marked simplicial category \(\textsc{Cart}\) spanned by pointed cartesian fibrations, with mapping \(\infty\)-category \(\textsc{Cart}_{\ast}(p,q)\) defined as the full subcategory of \(\textsc{Cart}(p,q)\) spanned (under the identification given by Lemma \ref{lem:mapping space of special cart fib}) by those maps \(f\colon X \to Y\) which preserve the terminal object. We let \(\mathbf{Cart}_{\ast}\) be the scaled simplicial nerve \(\mathrm{N}^{sc}\textsc{Cart}_{\ast}\).
\end{defi}
\begin{rmk}
Note that, since \(\textsc{Cart}\) is enriched over \(\infty\)-categories, we have that \(\textsc{Cart}_{\ast}\) is too, so that \(\mathbf{Cart}_{\ast}\) is an \(\infty\)-bicategory.
\end{rmk}
We can now prove the main result of this work.
\begin{thm}
	\label{thm:equiv}
There is an equivalence of \(\infty\)-bicategories between \(\mathbf{Cart}_{\ast}\) and \(\mathbf{Fact}_{cart}\), which sends an object \(p\colon \E \to \B\) in \(\mathbf{Cart}_{\ast}\) to \((\E,(p_{\mathcal{L}},p_{\mathcal{R}}))\), where \(p_{\mathcal{L}}\) is the class of maps inverted by \(p\) and \(p_{\mathcal{R}}\) is the class of \(p\)-cartesian morphisms.
\end{thm}
\begin{proof}
First, let us check that \((\E,(p_{\mathcal{L}},p_{\mathcal{R}}))\) is indeed an object of \(\textbf{Fact}_{cart}\). By hypothesis, \(\E\) has a terminal object, and by Proposition \ref{prop:fact sys from cart fib} \((p_{\mathcal{L}},p_{\mathcal{R}})\) is a factorization system on \(\E\). The two-out-of-three property for \(\mathcal{L}\) is trivially satisfied, so we are left with checking condition \((\ast)\) of Definition \ref{def:cartesian objects}. The first thing we have to check is that given a morphism \(g\colon pe \to px\), there is a \(p\)-cartesian morphism \(f\colon e' \to x\) and an equivalence \(g\simeq p(f)\) over \(px\). To get this, it is enough to pick \(f\) among the \(p\)-cartesian lifts of \(g\) with codomain \(x\).

Thanks to Remark \ref{rmk: cond *}, the next thing we have to check is that the map between the (homotopy) fibers of the vertical maps in the square \eqref{ff} over \(p\)-cartesian morphisms \(e\to x\) is an equivalence. Since the morphism \(g\) appearing there is in \(\mathcal{R}\), and is thus \(p\)-cartesian in the case at hand, that square is always a (homotopy) pullback, and therefore such comparison map between the fibers is necessarily an equivalence. This concludes the proof that  \((\E,(p_{\mathcal{L}},p_{\mathcal{R}}))\) is indeed an object in \(\textbf{Fact}_{cart}\).

The assignments on objects we have just described is essentially surjective, since given \((\E,(\L,\R))\) in \(\textbf{Fact}_{cart}\), we have that the associated localization functor \(L\colon \E \to L\E\) is a cartesian fibration which we denote by \(p\). Indeed, by assumption the induced maps \(\E_{/x}\to L\E_{/Lx}\) are localizations for every \(x\) in \(\E\), so \(p\) is a cartesian fibration thanks to Proposition \ref{prop:cart fib characterization}, and it obviously preserve the terminal object so it is pointed. The cartesian factorization system on \(\E\) induced by \(p\), which we denote by \((p_{\L},p_{\R})\), is equivalent to \((\L,\R)\) since, by Proposition \ref{prop: localization with 1}, we have that \(\L\) coincide with the class of maps inverted by \(L=p\), so \(\L=p_{\L}\) by definition. 

We can extend the assignment on objects to a map of marked simplicial categories thanks to the identification provided by Lemma \ref{lem:mapping space of special cart fib}, which provides a natural inclusion of the form \(\textsc{Cart}_{\ast}(p,q)\hookrightarrow \textsc{Cat}_{\infty}(\E,\C)\) for every pair of cartesian fibrations \(p\colon \E \to \B\) and \(q\colon \C\to\D\). The image of this inclusion can be identified with the full subcategory spanned by those maps \(f\colon \E \to \C\) which:
\begin{itemize}
\item preserve the terminal object, by assumption.
\item are such that \(f(p_{\mathcal{L}})\subset q_{\mathcal{L}}\), since by Lemma \ref{lem:mapping space of special cart fib} we have that \(q\circ f\) inverts all the maps in \(p_{\mathcal{L}}\).
\item  are such that \(f(p_{\mathcal{R}})\subset q_{\mathcal{R}}\), since morphisms of cartesian fibrations are required to preserve cartesian morphisms by definition.
\end{itemize}
Therefore, we have an equivalence of marked simplicial categories \(\textsc{Cart}_{\ast}\simeq \textsc{Fact}_{Cart}\), which induces the equivalence of \(\infty\)-categories in the statement upon application of the scaled simplicial nerve.
\end{proof}
\begin{obs}
By considering the fully faithful inclusion \(\mathrm{N}\colon \mathbf{Cat}\hookrightarrow \mathbf{Cat}_{\infty}\) we get the analogous statement of the previous theorem for ordinary categories.
\end{obs}
Finally, we record here the following corollary of the main theorem.
\begin{cor}
The reflection functor \(L\colon\E \to L\E\) associated with a cartesian factorization system \((\E,(\L,\R))\) is simple.
\end{cor}
\begin{proof}
Thanks to Theorem \ref{thm:equiv}, we know that \((\L,\R)\) must be induced by a cartesian fibration \(p\colon\E \to \B\). By Corollary \ref{cor:cart fact are simple}, we get the desired result.
\end{proof}
\bibliographystyle{amsplain}
\bibliography{biblio}

\providecommand{\bysame}{\leavevmode\hbox to3em{\hrulefill}\thinspace}
\providecommand{\MR}{\relax\ifhmode\unskip\space\fi MR }
\providecommand{\MRhref}[2]{%
  \href{http://www.ams.org/mathscinet-getitem?mr=#1}{#2}
}
\providecommand{\href}[2]{#2}
\begin{thebibliography}{1}

\bibitem{CisinskiHigherCats}
Denis-Charles Cisinski, \emph{Higher categories and homotopical algebra},
  Cambridge Studies in Advanced Mathematics, vol. 180, Cambridge University
  Press, Cambridge, 2019.

\bibitem{GagnaHarpazLanariEquiv}
Andrea Gagna, Yonatan Harpaz, and Edoardo Lanari, \emph{On the equivalence of
  all models for \((\infty,2)\)-categories}, Available at
  https://arxiv.org/abs/1911.01905.

\bibitem{GepnerHaugsengNikolausLax}
David Gepner, Rune Haugseng, and Thomas Nikolaus, \emph{Lax colimits and free
  fibrations in {$\infty$}-categories}, Doc. Math. \textbf{22} (2017),
  1225--1266.

\bibitem{JoyalNotesQCat}
André Joyal, \emph{Notes on quasi-categories}, Preprint, 2008.

\bibitem{JoyalQCatsApplications}
Andr{\'{e}} Joyal, \emph{The theory of quasi-categories and its applications},
  CRM Quaderns \textbf{45 (II)} (2008), 147--496.

\bibitem{LurieGoodwillie}
Jacob Lurie, \emph{{\((\infty, 2)\)}-categories and the {G}oodwillie {C}alculus
  {I}}, Preprint.
  \href{http://www.math.harvard.edu/~lurie/papers/GoodwillieI.pdf}{Available at
  author's website}.

\bibitem{HTT}
\bysame, \emph{Higher {T}opos {T}heory}, Annals of Mathematics Studies, vol.
  170, Princeton University Press, Princeton, NJ, 2009.

\bibitem{RosickyTholenFact}
J.~Rosick\'{y} and W.~Tholen, \emph{Factorization, fibration and torsion},
  Journal of Homotopy and Related Structures \textbf{355} (2003), no.~9,
  3611--3623.

\end{thebibliography}
\end{document}